\documentclass[12pt,a4paper]{amsart}
\usepackage[english]{babel}
\usepackage[applemac]{inputenc}
\usepackage[T1]{fontenc}
\usepackage{palatino}
\usepackage{amsfonts}
\usepackage{amsmath}
\usepackage{amssymb}
\usepackage{amsthm}
\usepackage[colorlinks = true,citecolor = black]{hyperref}
\title{On the Haar Shift representations of Calder\'on-Zygmund Operators}
\author{Tuomas Orponen}\thanks{The author is supported by the Finnish Centre of Excellence in Analysis and Dynamics Research.}
\address{Department of Mathematics and Statistics, University of Helsinki, P.O.B. 68, FI-00014 Helsinki, Finland}
\email{tuomas.orponen@helsinki.fi}
\newcommand{\E}{\mathbb{E}}
\newcommand{\R}{\mathbb{R}}

\newcommand{\N}{\mathbb{N}}

\newcommand{\Z}{\mathbb{Z}}
\newcommand{\tn}{\mathbb{P}}

\newcommand{\Sh}{\mathbb{S}}
\newcommand{\calD}{\mathcal{D}}

\newcommand{\1}{\mathbf{1}}
\newcommand{\supp}{\operatorname{supp}}
\newcommand{\dist}{\operatorname{dist}}

\theoremstyle{plain}
\newtheorem{thm}[equation]{Theorem}
\newtheorem{lemma}[equation]{Lemma}

\theoremstyle{definition}
\newtheorem{definition}[equation]{Definition}
\newtheorem{notation}[equation]{Notation}

\theoremstyle{remark}
\newtheorem{remark}[equation]{Remark}

\numberwithin{equation}{section}

\begin{document}

\begin{abstract} In connection with proving the $A_{2}$ conjecture in 2010, T. Hytönen obtained a representation of general Calder\'on-Zygmund operators in terms of simpler operators known as Haar shifts. In this note, we prove that the result is sharp in the sense that Haar shift representations of Hytönen's type are \emph{only} available for Calder\'on-Zygmund operators.
\end{abstract}

\maketitle

\section{Introduction}

Calder\'on-Zygmund operators are, by definition, bounded linear mappings $T \colon L^{2}(\R^{d}) \to L^{2}(\R^{d})$, which admit an integral representation
\begin{equation}\label{intRep} \langle Tf,g\rangle = \iint f(y)g(x)K(x,y) \, dx dy,  \end{equation} 
whenever $f,g \colon \R^{d} \to \R$ are bounded functions with compact and disjoint supports. The function $K \colon \R^{d} \times \R^{d} \setminus \{(x,y) : x = y\} \to \R$ is required to be a \emph{standard kernel}, which means that it satisfies the inequalities
\begin{equation}\label{std1} |K(x,y)| \lesssim \frac{1}{|x - y|^{d}}, \end{equation} 
and
\begin{equation}\label{std2} |K(x,y) - K(x',y)| + |K(y,x) - K(y,x')| \lesssim \frac{|x - x'|^{\delta}}{|x - y|^{\delta + d}} \end{equation}
for some $\delta > 0$ and for all $x,x',y \in \R^{d}$ with $x \neq y$ and $|x - x'| < |x - y|/2$. In 2010, T. Hytönen \cite{Hyt} found a remarkably useful way of representing all Calder\'on-Zygmund operators as averages of much simpler objects termed \emph{Haar shifts}, often denoted by $\Sh$. Postponing precise definitions to the next section, the representation of a Calder\'on-Zygmund operator $T$ with parameter $\delta > 0$ has the following form:
\begin{equation}\label{HRep} T = \E_{\omega} \left[\sum_{m,n \in \Z_{+}} 2^{-(m + n)\delta/2} \Sh_{m,n}^{\omega}\right]. \end{equation} 
Representation theorems usually come bundled with a more or less trivial converse, stating, in brief, that the result cannot be further generalized: for example, this is the case with Riesz and Radon-Nikodym representation, since every regular Borel measure gives rise to a continuous linear functional on $C_{0}(X)$, and measurable functions generate absolutely continuous measures in an obvious fashion. The existence of such a converse to the Hytönen representation was not discussed in \cite{Hyt}, the most likely reason being that the point was not interesting so far as the immediate applications of the representation were concerned. In this note, we settle the issue: it turns out that the right hand side of \eqref{HRep} always defines a Calder\'on-Zygmund operator. The proof is not quite as straightforward as in the Riesz and Radon-Nikodym cases: this is due to the fact that the individual shifts $\Sh_{m,n}^{\omega}$ are \textbf{not} Calder\'on-Zygmund operators, which means that the averaging element $\E_{\omega}$ is essential for the argument.

\section{Dyadic Grids and the Haar Shift Representation}

We start by defining the Haar shifts, and also some basic concepts necessary to understand them. Our exposition follows closely the ones given in \cite{HLMORSU} and \cite{HPT}.
\begin{definition}[Dyadic grids]\label{grid} A \emph{dyadic grid} is any collection $\calD$ of cubes $Q \subset \R^{d}$ such that
\begin{itemize} 
\item[(i)] The cubes in $\calD^{k} = \{Q \in \calD : \ell(Q) = 2^{k}\}$ partition $\R^{d}$ for any $k \in \Z$;
\item[(ii)] Every cube $Q \in \calD^{k}$ is the disjoint union of $2^{d}$ cubes in $\calD^{k - 1}$. These cubes are called the \emph{children} of $Q$, and their collection is denoted by $ch(Q)$.
\end{itemize}
\end{definition} 

The \emph{standard dyadic grid} $\calD_{0}$ is an example of a dyadic grid in the above sense. By definition, the collection $\calD_{0}$ consists of all cubes $Q \subset \R^{d}$ having the form $Q = 2^{k}[n_{1},n_{1} + 1) \times \ldots \times [n_{d}, n_{d} + 1)$ for some integers $k \in \Z$ and $n_{1},\ldots,n_{d} \in \Z$. We also write $\calD_{0}^{k} = \{Q \in \calD_{0} : \ell(Q) = 2^{k}\}$. 
\begin{definition}[Haar functions] Let $\calD$ be a dyadic grid, and let $Q \in \calD$ be a cube. A \emph{Haar function associated with $Q$} is any function $h_{Q}$ of the form
\begin{displaymath} h_{Q} = \sum_{Q' \in ch(Q)} c_{Q'}\1_{Q'}, \qquad c_{Q'} \in \R. \end{displaymath} 
\end{definition}

\begin{remark} Note that our definition of Haar functions is slightly unconventional, as one often requires these functions to have vanishing integral in addition to the previous hypotheses. We have no need for this assumption, however, so we omit it to avoid an unnecessary reduction of generality in Theorem \ref{main}.
\end{remark}

\begin{definition}[Haar shift operators] Let $\calD$ be a dyadic grid, and let $m,n \in \Z_{+}$. An operator $\Sh$ defined on $L^{2}(\R^{d})$ is a \emph{Haar shift on $\calD$ of complexity type $(m,n)$}, if $\Sh$ has the form
\begin{displaymath} \Sh f(x) = \sum_{Q \in \calD} \frac{1}{|Q|} \mathop{\mathop{\sum_{Q',Q'' \in \calD,\: Q',Q'' \subset Q}}_{\ell(Q') = 2^{-m}\ell(Q)}}_{\ell(Q'') = 2^{-n}\ell(Q)} \langle f,h^{Q'}_{Q''}\rangle h^{Q''}_{Q'}(x), \end{displaymath}
where $h^{Q'}_{Q''}$ and $h_{Q'}^{Q''}$ are Haar functions on $Q''$ and $Q'$, respectively, satisfying
\begin{displaymath} \|h^{Q'}_{Q''}h_{Q'}^{Q''}\|_{\infty} \leq 1. \end{displaymath}  
\end{definition}

\begin{notation} We will save some vertical space by using the abbreviation
\begin{displaymath} \mathop{\mathop{\sum_{Q',Q'' \in \calD,\: Q',Q'' \subset Q}}_{\ell(Q') = 2^{-m}\ell(Q)}}_{\ell(Q'') = 2^{-n}\ell(Q)} =: \sum_{Q',Q'' \subset Q}^{(m,n)} \end{displaymath}
We should also mention that the notation $A \lesssim B$, already used in the introduction, means that $A \leq CB$, where the constant $C > 0$ depends only on $d$ and, occasionally, the parameter $\delta$.
\end{notation}

\begin{definition}[Random dyadic grids] Let $\Omega = (\{0,1\}^{d})^{\Z}$ be the space of all doubly infinite sequences of binary $d$-tuples. If $Q \in \calD_{0}^{k}$ and $\omega \in \Omega$, define
\begin{displaymath} Q + \omega := Q + \sum_{j < k} \omega_{j}2^{j}. \end{displaymath}
Then, set $\calD_{\omega}^{k} := \{Q + \omega : Q \in \calD_{0}^{k}\}$ and $\calD_{\omega} := \bigcup_{k \in \Z} \calD_{\omega}^{k}$. The collection $\calD_{\omega}$ is a dyadic grid in the sense of Definition \ref{grid}, for any choice of $\omega \in \Omega$. The product probability measure on $\Omega$ is denoted by $\tn$, and the symbol $\E_{\omega}$ will always refer to expectation with respect to $\tn$.
\end{definition}

Now we are prepared to state Hytönen's result on representing general Calder\'on-Zygmund operators in terms of Haar shifts:
\begin{thm}\label{HRepT} Let $T$ be a Calder\'on-Zygmund operator with parameter $\delta > 0$. Then, for every $\omega \in \Omega$ and $m,n \in \Z_{+}$, there exists a Haar shift $\Sh_{m,n}^{\omega}$ on $\calD_{\omega}$ of complexity type $(m,n)$ such that the representation 
\begin{displaymath} \langle Tf,g\rangle = C_{T}\E_{\omega} \sum_{m,n \in \Z_{+}} 2^{-(m + n)\delta/2} \langle \Sh_{m,n}^{\omega}f,g\rangle  \end{displaymath} 
holds for smooth functions $f,g \colon \R^{d} \to \R$ with compact support. Here $C_{T} \geq 0$ is a constant depending on $T$. Moreover, the shifts $\Sh_{m,n}^{\omega}$ satisfy $\|\Sh_{m,n}^{\omega}\|_{L^{2} \to L^{2}} \leq 1$ for all $m,n \in \Z_{+}$ and $\omega \in \Omega$.
\end{thm}

\begin{proof} see \cite[Theorem 4.1]{HPT} or the original reference \cite[Theorem 4.2]{Hyt}. \end{proof}

\section{A Converse to Theorem \ref{HRepT}}

Suppose that $0 < \delta < 1$. Let for every $m,n \in \Z_{+}$ and $\omega \in \Omega$ be given a number $\lambda_{m,n}^{\omega} \in \R$ and a Haar shift operator $\Sh_{m,n}^{\omega}$ satisfying $|\lambda_{m,n}^{\omega}| \leq 2^{-(m + n)\delta}$ and $\|\Sh^{\omega}_{m,n}\|_{L^{2} \to L^{2}} \leq 1$. Then the formula
\begin{displaymath} \langle Sf,g\rangle := \E_{\omega} \sum_{m,n \in \Z_{+}} \lambda_{m,n}^{\omega} \langle\Sh_{m,n}^{\omega}f,g\rangle, \quad f,g \in L^{2}(\R^{d}), \end{displaymath}
certainly defines a bounded linear operator $S \colon L^{2}(\R^{d}) \to L^{2}(\R^{d})$. Moreover,
\begin{thm}\label{main} $S$ is a Calder\'on-Zygmund operator.
\end{thm}
In order to prove the theorem, we need to carry out the following tasks:
\begin{itemize}
\item[(i)] Find a kernel $K \colon \R^{d} \times \R^{d} \setminus \{x = y\} \to \R$ so that $S$ is represented by $K$.
\item[(ii)] Prove that $K$ satisfies the estimates \eqref{std1} and \eqref{std2} with the exponent $\delta$ defined above.
\end{itemize} 
Guessing the correct kernel is easy. Recalling the definition of Haar shifts, $K$ can be nothing else but
\begin{displaymath} K(x,y) := \E_{\omega} \sum_{m,n \in \Z_{+}} \lambda_{m,n}^{\omega} \sum_{Q \in \calD_{\omega}} \frac{1}{|Q|} \sum_{Q',Q'' \subset Q}^{(m,n)} h^{Q''}_{Q'}(x)h^{Q'}_{Q''}(y). \end{displaymath} 
With this definition, $K(x,y)$ makes sense for $x \neq y$. Indeed, note that the functions $h_{Q'}^{Q''}$ and $h_{Q''}^{Q'}$ are both supported on $Q$. Now, if $d(Q) < |x - y|$, this implies that either $h^{Q''}_{Q'}(x) = 0$ or $h^{Q'}_{Q''}(y) = 0$ for all $Q',Q'' \subset Q$. Using this and the normalization $\|h_{Q''}^{Q'}h_{Q'}^{Q''}\|_{\infty} \leq 1$ we obtain
\begin{displaymath} \left| \sum_{Q \in \calD_{\omega}} \frac{1}{|Q|} \sum_{Q',Q'' \subset Q}^{(m,n)} h^{Q''}_{Q'}(x)h^{Q'}_{Q''}(y) \right| \leq \mathop{\sum_{x \in Q \in \calD_{\omega}}}_{d(Q) \geq |x - y|} \frac{1}{|Q|} \lesssim \frac{1}{|x - y|^{d}}. \end{displaymath} 
This not only shows that $K(x,y)$ makes sense for $x \neq y$, but also proves the (easier) estimate \eqref{std1} for $K$. 

Next, let us check that $S$ admits the integral representation \eqref{intRep} for bounded compactly supported functions $f$ and $g$ with disjoint supports. Since the supports have positive distance, say $\dist(\supp f,\supp g) =: \epsilon > 0$, we have
\begin{align*} & \left| \iint \sum_{m,n \in \Z_{+}} \lambda_{m,n}^{\omega} \sum_{Q \in \calD_{\omega}} \frac{1}{|Q|} \sum_{Q',Q'' \subset Q}^{(m,n)} h^{Q''}_{Q'}(x)h^{Q'}_{Q''}(y)f(y)g(x) \, dx dy \right|\\
& \leq \|fg\|_{\infty} \sum_{m,n \in \Z_{+}} |\lambda_{m,n}^{\omega}| \int_{\supp f} \int_{\supp g} \mathop{\sum_{x \in Q \in \calD_{\omega}}}_{d(Q) \geq \epsilon} \frac{1}{|Q|} \, dx dy \lesssim_{\delta,\epsilon,f,g} 1 \end{align*}
independently of $\omega \in \Omega$, and this allows us to use Fubini's theorem:
\begin{align*} \langle Sf,g \rangle & = \iint f(y)g(x) \E_{\omega} \left[ \sum_{m,n \in \Z_{+}} \lambda_{m,n}^{\omega} \sum_{Q \in \calD_{\omega}} \frac{1}{|Q|} \sum_{Q',Q'' \subset Q}^{(m,n)} h^{Q''}_{Q'}(x)h^{Q'}_{Q''}(y) \right] dx dy\\
& = \iint f(y)g(x)K(x,y) \, dx dy.   \end{align*}  

\subsection{H\"older continuity of $K$}

It remains to prove the estimate \eqref{std2} for $K$. By symmetry, it suffices to establish
\begin{displaymath} |K(x,y) - K(x',y)| \lesssim \frac{|x - x'|^{\delta}}{|x - y|^{\delta + d}}  \end{displaymath} 
for $x,x',y \in \R^{d}$ with $x \neq y$ and $|x - x'| < |x - y|/2$. To this end, we need 
\begin{lemma} Fix $k \in \Z$ and $x \in \R^{d}$. Then 
\begin{displaymath} \tn[\{\omega : x \in Q \in \calD_{\omega}^{k} \text{ and } d(x,\partial Q) \leq \tau \ell(Q)\}] \lesssim \tau, \qquad \tau \geq 0. \end{displaymath} 
\end{lemma}

\begin{proof} By definition, the cubes in $\calD_{\omega}^{k}$ have the form $Q + \sum_{j < k} \omega_{j}2^{j}$, where $Q \in \calD_{0}^{k}$ and $\omega \in (\{0,1\}^{d})^{\Z}$. The sums $\omega = \sum_{j < k} \omega_{j}2^{j}$ are uniformly distributed in $[0,2^{k}]^{d}$, so the probability in question is simply\begin{displaymath} 2^{-kd} \cdot |\{y \in [0,2^{k}]^{d} : x \in Q + y, \: Q \in \calD_{0}^{k} \text{ and } d(x,\partial(Q + y)) \leq \tau\ell(Q)\}|. \end{displaymath} 
A simple geometric argument shows that this probability is the same for every $x \in \R^{d}$ and equals 
\begin{displaymath} |\{y \in [0,1]^{d} : d(y,\partial [0,1]^{d}) \leq \tau\}| = 1 - |[\tau,(1 - \tau)]^{d}| = 1 - (1 - 2\tau)^{d} \lesssim \tau. \end{displaymath}
\end{proof}

\begin{proof}[Proof of (\ref{std2})] Fix $x,x',y \in \R^{d}$ with $|x - x'| < |x - y|/2$. Choose $N \in \N$ so that $2^{-(N + 1)} < |x - x'|/|x - y| \leq 2^{-N}$. Write $E_{0} := \R^{d}$, and for $k \in \N \setminus \{0\}$ define
\begin{displaymath} E_{k} = \{\omega : x \in Q \in \calD_{\omega}^{i(k)} \text{ and } d(x,\partial Q) < 2|x - x'|\}, \end{displaymath}
where $i(k) \in \N$ is the smallest index such that $\ell(Q) \geq 2^{k}|x - x'|$ for $Q \in \calD_{\omega}^{i(k)}$. Since $d(x,\partial Q) < 2|x - x'| \leq 2^{-k + 1}\ell(Q)$ for $\omega \in E_{k}$ and $x \in Q \in \calD_{\omega}^{i(k)}$, the previous lemma implies that $\tn[E_{k}] \lesssim 2^{-k}$. Then define
\begin{displaymath} G_{k} = E_{k} \setminus \bigcup_{j = k + 1}^{\infty} E_{j}, \qquad k \geq 0. \end{displaymath}
Since $\sum \tn[E_{k}] < \infty$, we have $\tn[\limsup_{k \to \infty} E_{k}] = 0$, which means that $\tn$ almost every $\omega$ belongs to only finitely many of the sets $E_{k}$: thus $\tn$ almost every $\omega$ lies in $G_{k}$ for some $k \geq 0$ (recall that $E_{0} = \R^{d}$). Also note that $\tn[G_{k}] \lesssim 2^{-k}$. Now we decompose the expectation $\E_{\omega}$ in the definition of $K(x,y)$ using the sets $G_{k}$:
\begin{align} |K(x,y) & - K(x',y)|\notag\\
& = \left|\E_{\omega}\left[ \sum_{(m,n) \in \Z^{2}_{+}} \lambda_{m,n}^{\omega}\sum_{Q \in \calD_{\omega}} \frac{1}{|Q|}\sum_{Q',Q'' \subset Q}^{(m,n)} (h_{Q'}^{Q''}(x) - h_{Q'}^{Q''}(x'))h_{Q''}^{Q'}(y)  \right]\right|\notag\\
&\label{form10} \leq \sum_{k \in \N} \int_{G_{k}} \sum_{(m,n) \in \Z^{2}_{+}} |\lambda_{m,n}^{\omega}| \left|\sum_{Q \in \calD_{\omega}} \frac{1}{|Q|}\sum_{Q',Q'' \subset Q}^{(m,n)} (h_{Q'}^{Q''}(x) - h_{Q'}^{Q''}(x'))h_{Q''}^{Q'}(y) \right| \, d\tn[\omega].  \end{align} 
Fixing $k \in \N$, we now plan to estimate the $\tn$-integral over the set $G_{k}$. Let 
\begin{displaymath} m < N - k - k_{d}, \end{displaymath}
where $k_{d} \in \N$, to be chosen in a moment, is an integer depending only on the dimension $d$. Fix $\omega \in G_{k}$. We claim that
\begin{equation}\label{kernel} \sum_{Q',Q'' \subset Q}^{(m,n)} (h_{Q'}^{Q''}(x) - h_{Q'}^{Q''}(x'))h_{Q''}^{Q'}(y) = 0 \end{equation} 
for any $Q \in \calD_{\omega}$. Assume the contrary, and let $Q \in \calD_{\omega}$ such that \eqref{kernel} fails. Then $Q$ contains $y$ and at least one of the points $x,x'$. Since $|x - x'| < |x - y|/2$, this implies $\ell(Q) \gtrsim |x - y|$. Let $R \in \calD_{\omega}$ be the cube of side-length $\ell(R) = 2^{-m - 1}\ell(Q)$, which contains $x$. Then, as $m + 1 \leq N - k - k_{d}$, we have
\begin{displaymath} \ell(R) = 2^{-m - 1}\ell(Q) \geq 2^{k + k_{d} - N}\ell(Q) \gtrsim 2^{k_{d}} 2^{k - N}|x - y|. \end{displaymath} 
Now we fix $k_{d} \in \N$ so large that $\ell(R) \geq 2^{k + 1 - N}|x - y|$ above, and then
\begin{displaymath} \ell(R) \geq 2^{k + 1 - N}|x - y| \geq 2^{k + 1}|x - x'| \end{displaymath}
by the choice of $N \in \N$. Next, let $\tilde{Q}$ be the cube in $\calD^{i(k + 1)}_{\omega}$ that contains $x$. By definition of $\calD_{\omega}^{i(k + 1)}$, and the estimate for $\ell(R)$ above, we have $\ell(R) \geq \ell(\tilde{Q})$, and so $\tilde{Q} \subset R$. Since $\omega \notin E_{k + 1}$, we may infer that $d(x,\partial R) \geq d(x,\partial \tilde{Q}) \geq 2|x - x'|$. It follows that $x' \in R$. Now it suffices to note that the functions $h_{Q'}^{Q''}$ appearing in (\ref{kernel}) are constant on all cubes $R' \in \calD_{\omega}$ of side-length $\ell(R') = 2^{-m - 1}\ell(Q)$, including $R$, whence $h_{Q'}^{Q''}(x) = h_{Q'}^{Q''}(x')$  no matter how we choose $Q',Q'' \subset Q$. This proves \eqref{kernel}. Thus, for $\omega \in G_{k}$,
\begin{align*} \sum_{(m,n) \in \Z^{2}_{+}} & |\lambda_{m,n}^{\omega}| \left|\sum_{Q \in \calD_{\omega}} \frac{1}{|Q|}\sum_{Q',Q'' \subset Q}^{(m,n)} (h_{Q'}^{Q''}(x) - h_{Q'}^{Q''}(x'))h_{Q''}^{Q'}(y) \right|\\
& \leq \sum_{m \geq N - k - k_{d}} 2^{-m\delta} \sum_{n \in \Z_{+}} 2^{-n\delta} \left| \sum_{Q \in \calD_{\omega}} \frac{1}{|Q|}\sum_{Q',Q'' \subset Q}^{(m,n)} (h_{Q'}^{Q''}(x) - h_{Q'}^{Q''}(x'))h_{Q''}^{Q'}(y) \right|. \end{align*} 
As we saw while proving (\ref{std1}), the expression inside the absolute values can be bounded by a dimensional constant times $\max\{|x - y|^{-d},|x' - y|^{-d}\} \lesssim |x - y|^{-d}$. The summation over $n \in \Z_{+}$ only contributes a multiplicative constant depending on $\delta$, so we may continue the previous estimate by
\begin{displaymath} \ldots \lesssim \frac{1}{|x - y|^{d}} \sum_{m = N - k - k_{d}} 2^{-m\delta} \lesssim 2^{\delta k} \frac{2^{-N\delta}}{|x - y|^{d}} \lesssim 2^{\delta k} \frac{|x - x'|^{\delta}}{|x - y|^{\delta + d}}. \end{displaymath}
Substituting this estimate into \eqref{form10} and using $\tn[G_{k}] \lesssim 2^{-k}$, we obtain
\begin{displaymath} \sum_{k \in \N} \int_{G_{k}} \cdots \, d\tn[\omega] \lesssim \frac{|x - x'|^{\delta}}{|x - y|^{\delta + d}} \sum_{k \in \N}^{\infty} 2^{k(\delta - 1)} \lesssim \frac{|x - x'|^{\delta}}{|x - y|^{\delta + d}}, \end{displaymath} 
as claimed.
\end{proof}

This completes the proof of Theorem \ref{main}.

\end{document}